\documentclass[12pt]{amsart}
\usepackage[all]{xy}

\usepackage{graphicx}

\usepackage{amsmath}
\usepackage{amssymb}
\usepackage{amsfonts}
\usepackage{mathrsfs}
\usepackage{mathtools}
\usepackage[all]{xy}
\usepackage{color}
\usepackage[table]{xcolor}

\newcommand{\Cdb}{\ensuremath{\mathbb{C}}}

\newtheorem{theorem}{Theorem}[section]
\newtheorem{lemma}[theorem]{Lemma}
\newtheorem{corollary}[theorem]{Corollary}
\newtheorem{proposition}[theorem]{Proposition}
\newtheorem{definition}[theorem]{Definition}
\theoremstyle{remark}
\newtheorem{remark}[theorem]{\bf Remark}
\theoremstyle{definition}

\numberwithin{equation}{section}

\def\sqr#1#2{{\,\vcenter{\vbox{\hrule height.#2pt\hbox{\vrule width.#2pt
height#1pt \kern#1pt\vrule width.#2pt}\hrule height.#2pt}}\,}}
\def\bo{\sqr44\,}
\def\q{\quad}
\def\vp{\varphi}
\begin{document}

\title[holomorphic homogeneous regular domains]{}
\centerline{\bf\large Infinite dimensional holomorphic homogeneous}
\centerline{\bf\large regular domains}
\vspace{.2in}

\author[C-H. Chu, K-T. Kim, S. Kim]{}
\centerline{\footnotesize Cho-Ho Chu $\cdot$ Kang-Tae Kim $\cdot$ Sejun Kim}

\date{}
\thanks{\noindent \hspace{-.2in} C-H. Chu\\ School of Mathematical Sciences, Queen Mary, University of London,
London E1 4NS, UK\\
e-mail: c.chu@qmul.ac.uk}
\thanks{\noindent K.-T. Kim\\ Center for Geometry and its applications
and Department of Mathematics,
Pohang University of Science and Technology,
Pohang 790-784,
South Korea\\
e-mail: kimkt@postech.ac.kr}
\thanks{\noindent S. Kim\\ Center for Geometry and its applications
and Department of Mathematics,
Pohang University of Science and Technology,
Pohang 790-784,
South Korea\\
e-mail: rlatpwns21@postech.ac.kr}

\begin{abstract} We extend the concept of a finite dimensional {\it holomorphic homogeneous
regular} (HHR) domain and some of its properties to the infinite dimensional setting. In particular,
we show that infinite dimensional HHR domains are domains of holomorphy and determine completely
the class of infinite dimensional bounded symmetric domains which are HHR. We compute the greatest
lower bound of the squeezing function of all HHR bounded symmetric domains, including the two exceptional domains.
We also show
that uniformly elliptic domains in Hilbert spaces are HHR.
\end{abstract}

\maketitle

\section{Introduction}

The concept of a {\it holomorphic homogeneous regular} ({\it HHR}) complex manifold
$M$ of finite dimension has been introduced by Liu, Sun and Yau  \cite{lsy}
in connection with the estimation of several invariant metrics on the moduli
and Teichm\"uller spaces of Riemann surfaces.
It can be described by saying that a particular function
$\sigma : M \rightarrow (0,1]$, called the {\it squeezing function},  has
strictly positive lower bound (cf.\,\cite{deng}). These manifolds possess many important
geometric properties (e.g. all classical metrics on them are equivalent) \cite{lsy,lsy1}
and have also been studied by several
authors (see, for example, \cite{deng,deng1,fo,kim,yeung}) in the case of
complex domains.
In particular,  it has been shown in \cite{yeung} that  a holomorphic homogeneous
regular bounded  domain $D$
in $\mathbb{C}^n$ must be pseudoconvex and all strongly convex domains in $\mathbb{C}^n$
are holomorphic homogeneous regular.
Recently, it has been shown in \cite{kim} that all bounded convex
domains in $\mathbb{C}^n$ are holomorphic
homogeneous regular. The squeezing function on a bounded homogeneous
domain in $\mathbb{C}^n$ is constant, by its holomorphic invariance,
and has been computed explicitly for the four  classical series of Cartan
domains in \cite{k1}. In view of these interesting works, it is natural to ask
if they can be extended to the setting of infinite dimensional domains.

The object of this paper is to begin a study of infinite dimensional holomorphic homogeneous
regular domains. We extend  the concept of a
 holomorphic homogeneous regular domain and generalise the aforementioned results to the infinite
dimensional
setting. In addition, we also obtain new results in finite dimensions, in particular,
the squeezing functions are explicitly computed for all bounded symmetric domains,
including the two exceptional domains, which were left untreated in \cite{k1}.

The concept of the squeezing function for domains in $\mathbb{C}^n$
involves comparing a given domain  with various Euclidean balls via
embeddings. For infinite dimensional domains, we consider their
holomorphic embeddings in {\it Hilbert balls}, that is, open unit balls
of complex Hilbert spaces.

Throughout, all Banach spaces $V$ are over the complex field $\mathbb{C}$ and the dual of
$V$ is denoted by $V^*$.
Let $D$ be a bounded domain in a (complex) Banach space $V$.
We will call a map $f:D_1 \rightarrow D_2$ between two domains a
{\it holomorphic embedding} of $D_1$ in $D_2$ if $f(D_1)$ is a domain
in $D_2$ and $f$ is biholomorphic onto $f(D_1)$.

Let
$B_H =\{x\in H: \|x\| <1\}$ be the open unit ball of a Hilbert space $H$
and denote by
$H(D,B_H)$ the set of all holomorphic embeddings of $D$ into $B_H$, which may be
an empty set.  For instance, if $D$ is the open unit ball of the Banach space $\ell^\infty$ of bounded
sequences, then
$H(D, B_H)=\emptyset$ for any Hilbert ball $B_H$.

In fact, $H(D,B_H) \neq \emptyset$ if and only if the ambient Banach space
$V$ of $D$ is linearly homeomorphic to $H$.
Indeed,  if there is a holomorphic embedding $f:
D \rightarrow B_H$, then $V$, as the tangent space at a point $p$ in $D$, must be
linearly homeomorphic to $H$, which is the tangent space of $f(D)$ at $f(p)$.
Conversely,  if $\varphi : V \rightarrow H$ is a
linear homeomorphism, then  we have
$\varphi(D ) \subset RB_H$ for some $R >0$, and for each $p\in D$,
the map $f: z\in D \mapsto \varphi(z-p)/2R \in B_H$ is a
biholomorphic map onto the domain $f(D)$ in $B_H$, with
$f(p)=0$ and $rB_H \subset f(D) \subset B_H$ for some $r > 0$.

Given $H(D,B_H)\neq \emptyset$,
then for each $p \in D$, the set
$$
\mathcal{F}(p,D) = \{ f\in H(D,B_H):f(p)=0\}
$$
is nonempty, as noted previously. Hence we can
define the {\it squeezing function} $\sigma_D : D \rightarrow (0,1]$ by
$$
\sigma_D (p) =
 \sup_{f\in \mathcal{F}(p,D) } \{r>0: rB_H \subset f(D)\}.
$$
The {\it squeezing constant} $\hat\sigma_D$ for $D$ is defined by
$$
\hat\sigma_D = \inf_{p\in D} \sigma_D(p).
$$
Both the squeezing function and squeezing constant are biholomorphic invariants.

\begin{remark}\label{vh}
We note that, if $H(D,B_H) \neq \emptyset$, then the definition of the squeezing function for a domain
$D \subset V$ does not depend on the chosen Hilbert ball $B_H$.  Indeed, if there
is a holomorphic embedding of $D$ into another Hilbert ball $B_K$ of a Hilbert space $K$, then
the previous remarks imply that
there is a continuous linear isomorphism $T: H \rightarrow K$. Let
$\alpha : H^* \rightarrow H$ and $\beta: K \rightarrow K^*$ be
the canonical isometries. Then the linear isomorphism
$\alpha T^* \beta T: H \rightarrow H$ satisfies
$$
\langle \alpha T^* \beta T x, y\rangle_H
= \langle Tx,Ty\rangle_K \qquad (x, y\in H)
$$
and the linear isomorphism
$T(\alpha T^* \beta T)^{-1/2}: H \rightarrow K$ is an isometry.
It follows that the squeezing functions $\sigma_D$ defined in terms $B_H$ and $B_K$ respectively are
identical.
\end{remark}

We now extend the concept of
a finite dimensional HHR manifold introduced in \cite{lsy,lsy1} to infinite dimensional complex domains.
A finite dimensional HHR domain is also called a domain with {\it uniform squeezing property} in \cite{yeung}.

\begin{definition} \rm A  bounded domain $D$  in a complex Banach space $V$ is called
{\it holomorphic homogeneous regular} (HHR) if $D$ admits a holomorphic embedding into
 some Hilbert ball $B_H$ and  its squeezing function $\sigma_D: D
\rightarrow (0,1]$ has
a strictly positive lower bound, that is, $\hat \sigma_D >0$.
\end{definition}

\begin{remark} If $D$ is an HHR domain in a Banach space $V$, then as noted previously,
$V$ must be linearly homeomorphic to a Hilbert space.
We call $V$ an {\it isomorph of a Hilbert space}. The class of of these
Banach spaces has been characterised by many authors, for instance,
it has been shown in \cite{kw} that a Banach space is an isomorph of
a Hilbert space if and only if it is of type $2$ and cotype $2$.  We
refer to \cite[Chapter IV]{pel} for more details.\end{remark}

For infinite dimensional bounded {\it symmetric} domains, we shall
see that only those of finite rank can be embedded holomorphically in
a Hilbert ball.
We prove the following main results.\\

\noindent {\bf Theorem \ref{pseudo}}.
\noindent {\it An HHR domain is a domain of holomorphy}.\\

This result extends the finite dimensional result in \cite[Lemma 2]{yeung} since a domain of holomorphy in a
Banach space is pseudoconvex (cf.\,\cite[11.4, 37.7]{muji}). We note that a domain of holomorphy
need not be HHR even in finite dimensions, as shown in \cite[Theorem 1]{fo}.

 The following result reveals the connection
between the rank of a symmetric domain and the extent to which a
Hilbert ball can be squeezed inside it.\\

\noindent {\bf Theorem \ref{bd}}.
{\it Let $D$ be a bounded symmetric domain in a complex Banach space $V$. Then
$D$ is HHR if and only if $D$ is of finite rank. In this case,
 $D$ is
biholomorphic to a finite product
$$
D_1 \times \cdots \times D_k
$$
of irreducible bounded symmetric domains and we have
$$
\hat \sigma_D = \left( \frac{1}{\hat \sigma_{D_1}^2} + \cdots
+ \frac{1}{\hat \sigma_{D_k}^2}\right)^{-1/2}.
$$

If $\dim D_j <\infty$, then  $D_j$ is a classical Cartan domain or
an exceptional domain, and $\hat \sigma_{D_j}= 1/\sqrt{p_j}$
where $p_j$ is the rank of $D_j$.

If $\dim D_j =\infty$, then  $D_j$ is either a Lie ball or a type I
domain of finite rank $p_j$.  For a Lie ball $D_j$, we have
$\hat \sigma_{D_j} = 1/\sqrt 2$. For a rank $p_j$ type I domain $D_j$, we have
$\hat \sigma_{D_j} = 1/\sqrt{p_j}$.}\\

\noindent {\bf Theorem \ref{strongconv}}.
{\it Let $\Omega$ be a uniformly elliptic domain in a
Hilbert space $H$.  Then $\Omega$ is HHR.}\\

We introduce the concept of a uniformly elliptic domain in Section \ref{u}, which
generalises the notion of strong convexity.
This theorem generalises the finite dimensional result in
\cite[Proposition 1]{yeung}.

\section{Holomorphic homogeneous regular domains}

We begin our discussion of infinite dimensional HHR domains in this section by showing some properties of
the squeezing function and conclude with a proof of pseudoconvexity for these domains.

Given two (nonempty) sets $A$ and $B$ in  a Banach space $V$, we write
$$d(A,B ) = \inf \{ \|x-y\| : x \in A, y \in B\}$$
and for $p\in V$, write $d(p, B)$ for $d(\{p\}, B)$  which is the distance from $p$ to $B$.
Let $D$ be a bounded domain in $V$, a closed subset $K$ of $D$ is
said to be {\em strictly contained} in $D$ if $d(K, V  \backslash D)>0$.
Let $$B_V (p, r) := \{ z \in V \colon \|z-p\|<r\}$$
denote the norm-open ball centred at $p$ with radius $r>0$.  The open unit ball $B_V(0,1)$ is often written simply
$B_V$. We will make use of the Carath\'eodory
distance $C_D$ on $D$, which is equivalent to the norm-distance on any closed ball (for the norm)
strictly contained in $D$ (see \cite[Theorem IV.2.2]{fv}).
For each $v\in B_V$, we have
$C_{B_V}(v,0) = \tanh^{-1} \|v\|$, by \cite[Theorem IV.1.8]{fv}.

In what follows, the boundary of a topological subspace $E$ of $V$ will be denoted by $\partial E$.
The complement of $E$ in $V$ will be denoted by $E^c$ and as usual, $\overline E$ denotes the closure
of $E$.

We first show that the squeezing function is continuous. Our proof follows the arguments in \cite[Theorem 3.1]{deng}.
It is included for completeness.

\begin{proposition} \label{con} Let $D$ be a bounded domain in a Banach space $V$ linearly homeomorphic
 to a Hilbert space $H$. Then the squeezing function $\sigma_D : D \rightarrow (0,1]$ is continuous.
 \end{proposition}

\begin{proof} Let $(z_k)$ be a sequence converging to $a\in D$. We show
$$\lim_{k\rightarrow\infty}\inf \sigma_D(z_k) \geq
\sigma_D(a) \geq \lim_{k\rightarrow\infty}\sup \sigma_D(z_k).$$
Let $0< 2\varepsilon <\sigma_D(a)$ and pick $\sigma_D(a) \geq \rho > \sigma_D(a) - \varepsilon$ such that there is a holomorphic
embedding $f: D \rightarrow B_H$ satisfying $f(a)=0$ and $\rho B_H \subset f(D)$. By continuity, we have
$$\|f(z_k)\| <\varepsilon$$
 for $k >K$, for some $K>0$. Consider the holomorphic embedding $f_k : D \rightarrow B_H$ given by
$$f_k(\omega) = \frac{f(\omega) - f(z_k)}{1+ \varepsilon} \qquad (\omega \in D)$$
which satisfies $f_k(z_k) = 0$ and
$$\frac{\rho - \varepsilon}{1+\varepsilon} B_H \subset f_k(D).$$
This gives $$\sigma_D (z_k) \geq \frac{\rho-\varepsilon}{1+\varepsilon}
> \frac{\sigma_D(a)-2\varepsilon}{1+\varepsilon} $$
 for $k>K$ and hence $\lim_{k\rightarrow\infty}\inf \sigma_D(z_k) \geq
\sigma_D(a)$ since $\varepsilon >0$ was arbitrary.

For the upper limit, let $0< 2\varepsilon < \lim_{k}\inf \sigma_D(z_k) $ and let $f_k: D \rightarrow B_H$ be a holomorphic
embedding satisfying $f_k(z_k)=0$ and $\rho_k B_H \subset f_k(D)$ for some $\sigma_D(z_k) \geq \rho_k > \sigma_D(z_k) -\varepsilon$.
Since $C_D(z_k,a) \rightarrow 0$ as $k\rightarrow \infty$, we have
$$\tanh^{-1} \|f_k(a)\|= C_{B_H}(0,f_k(a))\leq C_D(0,a) \rightarrow 0$$ and hence there exists some $M>0$
such that $\|f_k(a)\| < \varepsilon$ for $k>M$. By analogous arguments as before, one obtains
$$\sigma_D(a) \geq  \frac{\rho_k-\varepsilon}{1+\varepsilon}
> \frac{\sigma_D(z_k)-2\varepsilon}{1+\varepsilon}$$
for $k >M$, which gives $\sigma_D(a) \geq \lim_{k\rightarrow\infty}\sup \sigma_D(z_k).$
\end{proof}

Although the continuity of the squeezing function implies readily that if there is a sequence
$(p_k)$ in a {\it finite dimensional} bounded domain $D$ with $\lim_k \sigma_D(p_k) =0$, then the sequence
admits a subsequence $(p_j)$ converging to a boundary point $p\in \partial D$, this is not immediately clear
for infinite dimensional domains. Nevertheless, one can still show, in infinite dimension,
$(p_k)$ has a subsequence $(p_j)$ for which the distance $d(p_j, \partial D)$ to the boundary
tends to $0$. We prove a lemma first.

\begin{lemma}\label{rr} Let $\Omega$ be a bounded  domain in an isomorph $V$ of
a Hilbert space $H$ and $\vp: V \rightarrow H$ a linear homeomorphism.
Then there is a constant $m>0$ such that for each $q\in \Omega$ satisfying $B_V(q,s) \subset \Omega$ for some $s>0$,
we have $$\sigma_\Omega (q)\geq \frac{s}{m^2\|\vp\|\|\vp^{-1}\|}.$$
\end{lemma}
\begin{proof}
By a translation, we may assume $q=0$. Since $\Omega$ is bounded, we have
$\Omega \subset B_V(0, m)$ for some $m>0$ and
\begin{equation}\label{subset}
\frac{1}{\|\vp^{-1}\|} B_H(0,m) \subset \vp (B_V(0,m)) \subset B_H(0,m\|\vp\|) = m\|\vp\|B_H.
\end{equation}
The restriction of $\vp$ to $\Omega$, still denoted by $\vp$, is a holomorphic embedding
of $\Omega$ into $m\|\vp\|B_H$ satisfying $\vp (q)=0$.  It follows from (\ref{subset})
that
$$ \frac{s}{m\|\vp^{-1}\|}B_H(0,m) \subset \vp (B_V(0,s))\subset \vp (\Omega) \subset \vp(B_V(0,m))\subset m\|\vp\|B_H.$$
Hence we have
$$\sigma_\Omega(q) \geq \frac{s}{m^2\|\vp\|\|\vp^{-1}\|}.$$
\end{proof}

\begin{lemma}\label{boundary}  Let $(p_k)$ be a sequence in a  bounded convex domain $\Omega$ in an
isomorph $V$ of a Hilbert space such that $\lim_{k\rightarrow \infty} \sigma_\Omega(p_k)=0$. Then there is a subsequence $(p_j)$ of
$(p_k)$ such that $$\lim_{j\rightarrow \infty} d(p_j, \partial \Omega) =0.$$
Further, there is a sequence $(p_j')$ in $\Omega$ such that
$\lim_{j\rightarrow \infty} \sigma_\Omega(p_j')=0$,
and for each $j$, there exists a boundary point $q_j\in \partial \Omega$
with $\|p_j'-q_j\|= d(p_j', \partial\Omega)$.
\end{lemma}

\begin{proof} Let $(p_k)$ be the given sequence satisfying
\begin{equation}\label{zero}
\lim_{k\rightarrow \infty} \sigma_\Omega(p_k) =0.
\end{equation}
Since the bounded domain $\Omega$ is relatively weakly compact in  $V$, there is a subsequence $(p_j)$ in
$\Omega$ converging weakly to some point $p\in \overline \Omega$. We do not know if the squeezing function
$\sigma_\Omega$ is weakly continuous on $\Omega$.

Let $r_j = d(p_j, \partial \Omega)$ denote the distance from $p_j$ to the boundary $\partial \Omega$.
We first show that $\lim_{j \rightarrow \infty} r_j = 0$. Otherwise, we may assume (by choosing a subsequence)
$$r_j \geq s, \quad {\rm for ~some} \quad s >0$$
for all $j$. For all $z\in \partial \Omega$, we have $\|z-p_j\| \geq r_j$. Observe that
$B_V(p_j, r_j) \subset \Omega$, for if there exists some $\omega \in B_V(p_j, r_j) \backslash \Omega$,
then we must have $\omega \notin \overline\Omega$. Therefore the (real) line joining $p_j$ and $\omega$
must intersect $\partial \Omega$ at a point $z_0$ say, which gives a contradiction that
$$r_j \leq \|z_0 - p_j\| \leq \| \omega - p_j\| <r_j.$$
By Lemma \ref{rr}, there exists $m>0$ such that
$$\sigma_\Omega(p_j) \geq \frac{r_j}{m^2\|\vp\|\|\vp^{-1}\|} \geq \frac{s}{m^2\|\vp\|\|\vp^{-1}\|} >0,$$
contradicting  $\lim_j \sigma_\Omega(p_j) =0$. Therefore we have established
$$r_j = d(p_j, \partial \Omega) \rightarrow 0 \quad {\rm as} \quad j \rightarrow \infty.$$

To show the second assertion, we make use of a result in  \cite[Theorem 3.2]{b}, which states that in a reflexive Banach space $V$,
if the complement $V \backslash C$ of a non-empty closed set $C$ in $V$ is convex, then
$C$ is almost proximinal, in other words,
there is a dense $G_\delta$
set $A$ in $V \backslash C$ such that for each $x\in A$, there is a point $z\in C$  satisfying
$$\|x-z\|= d(x, C).$$

The given Banach space $V$ is reflexive.
We apply the above result to the set $C= V\backslash \Omega$, which is almost proximinal.
By continuity of the squeezing function $\sigma_\Omega$, there is an open neighbourhood
$N_j$ of $p_j$ such that $\sigma_\Omega(x) < 2\sigma_\Omega (p_j)$ for all $x\in N_j$
and for each $j$.
By density of $A$, one can find $p_j'\in A\cap N_j$ for which there exists $q_j \in C $ satisfying
$$\|p_j' - q_j\| = d(p_j', C)  =d(p_j', V\backslash \Omega) \leq d(p_j', \partial \Omega)$$
where the last inequality holds because the boundary $\partial \Omega$ is contained in $V\backslash \Omega$.

If $q_j \notin \partial \Omega$, then $q_j \notin \overline \Omega$ since  $q_j \notin \Omega$.
Hence the line segment $\{p_j' + \alpha (q_j - p_j')
: 0\leq \alpha\leq 1\}$ joining $p_j'$ and $q_j$ must intersect the boundary $\partial \Omega$
at some point $\omega
=p_j' + \beta (q_j - p_j')\in \partial \Omega$ with $0<\beta < 1 $.
It follows that
$$  \|p_j'-q_j\| \leq d(p_j',\partial \Omega) \leq \|p_j'-\omega\| = \beta\|p_j'-q_j\| < \|p_j'-q_j\|$$
which is impossible. Hence we have $q_j \in \partial \Omega$ and $\|p_j'-q_j\|= d(p_j', \partial \Omega)$.

 Finally, $\sigma_\Omega(p_j') < 2\sigma_\Omega(p_j)$
for all $j$ implies $\lim_j \sigma_\Omega(p_j') =0$.
\end{proof}

To show that an HHR domain $D$ in a complex Banach space $V$ is pseudoconvex, we  show that $D$ is a domain of holomorphy,
as defined in \cite{muji}. In finite dimensions, a bounded domain $D$ is a domain of holomorphy
if $(D, C_D)$ is complete \cite[p.368]{ko1}. We first extend this useful result to infinite dimension.

\begin{lemma}\label{24}
Let $D$ be a bounded domain in a complex Banach space $V$. If $D$ is complete with respect to the Carath\'eodory
distance, then it is a domain of holomorphy.
\end{lemma}

\begin{proof}
Suppose that $D$ is not a domain of holomorphy.  Then by definition, there are
open subsets $U, W$ of $V$ satisfying the following conditions:
\begin{enumerate}
\item[(i)] $U$ is connected.
\item[(ii)] $D \cap U \neq \emptyset$.
\item[(iii)] $U \not\subset D$.
\item[(iv)] $\emptyset \neq W \subset D \cap U$.
\item[(v)]  For each holomorphic function $f\colon D \to \mathbb{C}$, there is a holomorphic
function $\tilde f\colon U \to \mathbb{C}$ such that \(\tilde f (z) = f(z)\) for each
$z \in W$.
\end{enumerate}

We deduce a contradiction. Without loss of generality, we may assume that $U$ is bounded.
Let \(W_0\) be a connected component of $U \cap D$ with $W_0 \cap W \neq \emptyset$.
Then we have $$\partial W_0 \cap \partial D \cap U \neq \emptyset.$$ Indeed, if this is not
the case, then for each $p\in U \backslash W_0 \neq \emptyset$, either $p \notin \partial W_0$ or
$p\notin \partial D$. If $p \notin \partial W_0$, then there is a norm-open ball $B_p \subset U$ containing
$p$ such that either $B_p \cap W_0 = \emptyset$ or $B_p \cap W_0^c = \emptyset$. Since $p \notin W_0$,
we must have $B_p\cap W_0 =\emptyset$. On the other hand, if $p\notin \partial D$, then there is an
open ball $B_p$ containing $p$ such that $B_p \cap D = \emptyset$ or $B_p \cap D^c = \emptyset$.
In either case, we have $B_p \cap W_0 = \emptyset$ since, if $B_p \subset D$, then the connected ball
$B_p$ resides in a connected component $W_1$ of $U\cap D$ and we must have $W_1 \neq W_0$ as $p \notin W_0$.
Now the disconnection
$$U = W_0 \cup \left( \bigcup_{p\in U\backslash W_0} B_p\right)$$
contradicts the connectedness of $U$.

Pick a point $p \in \partial W_0\cap \partial D \cap U$ and
let $(z_n)$ be a sequence in $W_0$ norm-converging to $p$. By omitting the first few terms of the sequence if necessary,
we may assume that $(z_n)$ and $p$ are contained in a closed ball strictly contained in $U$. It follows that
$(z_n)$ also converges to $p$ with respect to the Carath\'eodory distance $C_U$.

By condition (v) above,
each holomorphic function $f\colon D \to \Cdb$ with $|f(z)|<1$
extends to a holomorphic function  $\tilde f: U \rightarrow \mathbb{C}$, which coincides with $f$
on the connected component $W_0$ by the identity principle. Moreover,
if $|\tilde f(u)|>1$ for some $u\in U$, then we deduce a contradiction by considering the extension  to
$U$ of the function $\frac{1}{f- \tilde f(u)}$ on $D$. Hence we must have $|\tilde f(u)| \leq 1$ for all $u\in U$
and, by the maximum principle, $|\tilde f(u)| <1$ for all $u\in U$. It follows that
\[
C_D(z_n,z_m)   \le  C_U(z_n,z_m)
\]
for $n,m =1,2,\ldots$, where $C_U (z_n,z_m)$ converges to $C_U (p,p)=0$ as $n,m \to \infty$. Hence
 $(z_n)$ is a Cauchy sequence in $D$ with respect to $C_D$.
However, $(z_n)$ does not converge in $D$, with respect to $C_D$. Indeed, if $(z_n)$
$C_D$-converges to some point $z\in D$ say, then
by \cite[Lemma 2.1]{fv}, there is a constant $\alpha >0$ such that
$$\alpha \|z_n - z\| \leq C_D(z_n,z) \rightarrow 0 \quad {\rm as} \quad n \rightarrow \infty$$
which is impossible since $(z_n)$ does not converge in $D$ with respect to the
norm-distance.  This shows   that
$(D, C_D)$ fails to be complete, which is a contradiction. We therefore conclude that $D$ is
a domain of holomorphy.
\end{proof}

We now extend the result in \cite[Lemma 2]{yeung} to the following infinite dimensional setting.

\begin{theorem}\label{pseudo}
Let $D$ be an HHR domain in a complex Banach space $V$.
Then $D$ is a domain of holomorphy.
\end{theorem}

\begin{proof}
In view of Lemma \ref{24}, we need only show that the Carath\'eodory distance in $D$ is
complete.

By the hypothesis, the squeezing constant $\hat \sigma_D$ takes the value,
say, $r\in (0,1]$.
Let $(x_n)$ be a $C_D$-Cauchy sequence in $D$. We show that $(x_n)$ $C_D$-converges.

Let $\varepsilon = \tanh^{-1} \frac{r}2$. Then there is a number $N>0$
such that
\(
C_D (x_n, x_N) < \varepsilon ~{\rm for}~  n >N.
\)

Let $f\colon D \to B_H$ be a holomorphic embedding into a Hilbert ball $B_H$ with
\( f(x_N)=0\) and \(B_H(0, \frac{3r}4) \subset f(D)\).  Then the inverse
holomorphic map $g:=f^{-1}\colon f(D) \to D$ is well-defined on the
ball $B_H (0, \frac{3r}4)$.

We have, for $n >N$,
\[
C_{B_H} (0,f(x_n)) = C_{B_H}(f(x_N), f(x_n)) \leq C_D (x_N, x_n) < \varepsilon = \tanh^{-1} \frac{r}2
\]
as well as
\[
\lim_{n, m \to \infty} C_{B_H} (f(x_m), f(x_n)) \le \lim_{n, m \to \infty}C_D(x_m, x_n) = 0.
\]
Since $B_H$ is complete in  the Carath\'eodory distance, there is a subsequence $(x_{n_k})$ of $(x_n)$
such that $f(x_{n_k})$ converges to some
$y_0 \in B_H$ with respect to $C_{B_H}$, and $C_{B_H}(0, y_0) \leq \varepsilon$.  Hence, as noted previously, we have
 $y_0 \in \overline{B}_H (0,\frac{r}2) \subset
B_H (0,\frac{3r}4) \subset f(D)$ and also,
\begin{align*}
\lim_{k \to \infty} C_D (x_{n_k}, g(y_0))
& \le \lim_{k \to \infty} C_D (g(y_{n_k}), g(y_0)) \\
& \le \lim_{k \to \infty}\frac{4}{3r} C_{B_H} (f(g(y_{n_k})), f(g(y_0))) \\
& = \lim_{k \to \infty}  \frac{4}{3r} C_{B_H} (y_{n_k}, y_0) = 0.
\end{align*}
It follows that the sequence $(x_n)$ converges to $g(y_0)$ in $D$ with respect to $C_D$ and the
proof is complete.
\end{proof}

In the remaining sections, we will show that various infinite dimensional domains are HHR, including the
finite-rank bounded symmetric domains and the class of strongly convex domains in Hilbert spaces.

\section{Bounded symmetric domains}

In this section, we discuss infinite dimensional bounded symmetric domains and some basic
results which are needed later. We will make use of the underlying Jordan algebraic structures
of a bounded symmetric domain to study the squeezing function.

Let $D$ be a bounded symmetric domain in a complex Banach space $V$.
Then $V$ carries the structure of a {\it JB*-triple},
by Kaup's Riemann mapping theorem \cite{kaup2} (see also
\cite[Theorem 2.5.26]{book}).
More precisely, $V$ is equipped with an {\it equivalent norm} $\|\cdot\|$
and a continuous {\it Jordan triple product}
$$
\{\cdot,\cdot,\cdot\} : V \times V \times V \rightarrow V
$$
which is linear in the outer variables but conjugate linear in the middle one,
and satisfies the following conditions:
\begin{enumerate}
\item[(\romannumeral 1)] $\{u,v, \{x,y,z\}\}
= \{\{u,v,x\},y,z\} - \{x,\{v,u,y\},z\} + \{x,y,\{u,v,z\}\}$;
\item[(\romannumeral 2)] $z\bo z : v\in V \mapsto \{z,z,v\} \in V$ is
a hermitian operator on $V$, that is, $\|\exp it(z\bo z)\|=1$ for all
$t \in \mathbb{R}$;
\item[(\romannumeral 3)]
$z\bo z$ has non-negative spectrum; \item[(iv)] $\|z\bo z\|=\|z\|^2$
\end{enumerate}
for all $u,v,x,y,z \in V$.  In this case, $D$ is biholomorphic to the open
unit ball $\{v\in V: \|v\|<1\}$ of the JB*-triple $(V, \|\cdot\|)$ and, we say that
$D$ is realised as the open unit ball of the JB*-triple $(V, \|\cdot\|)$.

The {\it rank} of $D$ can  be defined in terms of the Jordan structures of $V$.
A closed subspace $E$ of a JB*-triple $V$ is called a {\it subtriple} if
$a,b,c \in E$ implies $\{a,b,c\} \in E$.  For each $a\in V$, let $V(a)$ be
the smallest subtriple of $V$ containing $a$. For $V\neq \{0\}$,
the  {\it rank} of $V$ is defined to be
$$
r(V) = \sup \{ \dim V(a): a\in V\} \in \mathbb{N}\cup \{\infty\}.
$$
The {\it rank} of $D$ is defined to be $r(V)$.
A (nonzero) JB*-triple $V$ has {\it finite rank}, that is, $r(V) < \infty$ if,
and only if, $V$ is a reflexive Banach space (see \cite[Proposition 3.2]{kaup1}).
In particular, if there is a holomorphic embedding of $D$ into a Hilbert
ball, then $V$ is linearly homeomorphic to a Hilbert space and hence $D$
must be of finite rank.

A finite-rank JB*-triple can be {\it coordinatised} by elements called
{\it tripotents}. An element $e$ in a JB*-triple $V$ is called a {\it tripotent} if
$\{e,e,e\}=e$.
A nonzero tripotent $e$ is called {\it minimal} if $\{e, V, e\} = \mathbb{C}\,e$.
The Banach subspace $K_0(V)$ of $V$ generated by the minimal tripotents
has been studied in \cite{cb}.
Two elements $a, b \in V$ are said to be mutually (triple) {\it orthogonal} if
$a \bo b = b \bo a =0$, where $a\bo b$ denotes the
continuous linear operator
\[ a \bo b : x\in V \mapsto \{a,b,x\} \in V.\]
In fact, it can be shown that
$a\bo b=0$ is equivalent to $b\bo a=0$ \cite[Lemma 1.2.32]{book}.
For a finite-rank JB*-tiple $V$,  its rank $r(V)$ is the (unique) cardinality of a
maximal family of mutually orthogonal minimal tripotents in $V$, which is an
$\ell^{\infty}$-sum of a finite number of finite-rank Cartan factors.
There are six types of finite-rank Cartan factors, which
can be infinite dimensional, listed below.
$$
\begin{aligned}
\text{\rm Type I} &\q {L}(\mathbb{C}^\ell,K)\q (\ell =1,2, \ldots),~ {\rm rank}\,= \ell\leq \dim K,\q \\
\text{\rm Type II} &\q
\{z\in {L}(\mathbb{C}^\ell,\mathbb{C}^\ell): z^t=-z\}\q (\ell = 5,6,  \ldots),~ {\rm rank}\,=\left[\frac{\ell}{2}\right]\\
\text{\rm Type III} &\q  \{z\in
{L}(\mathbb{C}^\ell,\mathbb{C}^\ell): z^t=z\} \q (\ell = 2,3,  \ldots),~ {\rm rank}\,= \ell\\
\text{\rm Type IV} &\q \text{\rm spin
factor,}~ {\rm rank}\,= 2\\ \text{\rm Type V} &\q M_{1,2}(\mathbb{O}) = \{1\times
2 ~\,\text{\rm matrices over the Cayley algebra}\; \mathbb{O}\},~ {\rm rank}\,= 2\\
\text{\rm Type VI} &\q H_3(\mathbb{O})=\{3\times 3 ~\,\text{\rm
hermitian matrices over}\; \mathbb{O}\},~ {\rm rank}\,= 3
\end{aligned}
$$
where ${L}(\mathbb{C}^\ell,K)$ is the JB*-triple of linear
operators from $\mathbb{C}^\ell$ to a Hilbert space $K$ and
$z^t$ denotes
the transpose of $z$ in the JB*-triple ${L}(\mathbb{C}^\ell,\mathbb{C}^\ell) $
of $\ell \times \ell$ complex matrices. The Jordan triple product in the first
three types is given by
$$
\{x,y,z\} = \frac{1}{2}(xy^*z + zy^*x)
$$
where $y^*$ denotes the adjoint of $y$.

A {\it spin
factor} is a JB*-triple $V$  equipped with a complete inner
product $\langle \cdot,\cdot \rangle$ and a conjugation
$* : V \rightarrow V$ satisfying
 $$
 \langle x^*, y^*\rangle = \langle y, x\rangle \quad {\rm and} \q  \{x,y,z\} =\frac
12\,\big(\langle x,y\rangle z + \langle z,y\rangle x - \langle x,z^*\rangle y^*\big).
$$
The Cartan factor $H_3(\mathbb{O})$ is a Jordan algebra with product
\[x\cdot y= \frac{1}{2}(xy+yx)\]
where the product on the right-hand side is the usual matrix product.
The Jordan triple product of $H_3(\mathbb{O})$ is given by
\[\{x,y,z\} = (x\cdot y)\cdot z + x\cdot(y\cdot z) -y \cdot(x\cdot z).\]
The Cartan factor $M_{1,2}(\mathbb{O})$ can be identified as a subtriple of
$H_3(\mathbb{O})$.

The only possible infinite dimensional finite-rank Cartan factors are
the spin factors and $L(\mathbb{C}^{\ell},K)$, with $\dim K = \infty >\ell$,
where a spin factor has rank 2 and $L(\mathbb{C}^\ell,K)$ has rank $\ell$.
The open unit balls of the finite dimensional Cartan factors are exactly the
six types of irreducible bounded symmetric
domains in \'E. Cartan's classification. The last two types are the
exceptional domains. This explains the etymology of {\it Cartan factor}.
The open unit ball of a spin factor is known as a {\it Lie ball}.

For a finite-rank JB*-triple $V$ with rank $\ell$, each element $z\in V$ has
a {\it spectral decomposition}
$$
z = \alpha_1 e_1 + \cdots + \alpha_{\ell} e_{\ell}
$$
where $e_1, \ldots, e_\ell$ are mutually (triple) orthogonal
minimal tripotents and
$\alpha_1 \geq\cdots \geq \alpha_{\ell}\geq 0$ with
$\alpha_1 =\|z\|$, also called the {\it spectral norm} of $z$.

\section{Squeezing functions of bounded symmetric domains}

In finite dimensions, it is well-known that a bounded symmetric
domain of rank $\ell$ contains a polydisc of dimension $\ell$ as a
totally geodesic submanifold \cite[p.41]{ko}. To see that this is also the
case for infinite dimensional bounded symmetric domains of finite rank,
we only need to consider the {\it irreducible} ones. As remarked previously,
there are only two classes of such domains, namely,
the Lie balls, which are of rank $2$, and the type I domains of rank $\ell$, which can be
realised as the open unit ball of the Banach space $L(\mathbb{C}^\ell, K)$
of bounded linear operators between Hilbert spaces $\mathbb{C}^\ell$
and $K$, with  $\ell \leq \dim K \leq  \infty$ and $\ell <\infty$.

Given $\ell  < \infty$, every operator $T \in L(\mathbb{C}^\ell,K)$ is
a Hilbert-Schmidt operator in the Hilbert-Schmidt norm
$$
\|T\|_2 = (\sum _{k=1}^\ell \|Te_k\|^2)^{1/2}
$$
satisfying $\|T\| \leq \|T\|_2 \leq \sqrt\ell \|T\|$,
where $\{e_1, \ldots, e_\ell\}$ is the standard orthonormal basis
in $\mathbb{C}^\ell$.

Let  $\overline{\mathbb{D}}$ be the closure of $\mathbb{D}=\{z\in \mathbb{C}: |z| <1\}$ and
$\overline D$
the closure of the open unit ball
$$
D = \{T\in L(\mathbb{C}^\ell,K): \|T\| < 1\}.
$$
Fix orthonormal basis vectors $u_{\alpha_1}, \ldots, u_{\alpha_\ell}$
from an orthonormal basis $\{u_\alpha\}$ in $K$.
Then the continuous map
$\varphi : \overline{\mathbb{D}}\times \cdots \times \overline{\mathbb{D}}
\rightarrow \overline D$, defined by
\begin{equation}
\label{vp}
\varphi(  z_1, \ldots z_\ell)
= \sum _{k=1} ^\ell z_k(e_k \otimes u_{\alpha_k})
\qquad (z_1, \ldots, z_\ell) \in \overline{\mathbb{D}}^\ell,
\end{equation}
restricts to an injective holomorphic map
$$
\varphi : \mathbb{D}\times \cdots \times \mathbb{D} \rightarrow D
$$
with $\varphi(0, \ldots, 0) = 0$, where
$e_k \otimes u_{\alpha_k} : \mathbb{C}^\ell \rightarrow K$ is the
rank-one operator
$$
e_k \otimes u_{\alpha_k}(h) = \langle h, e_k\rangle u_{\alpha_k}
\qquad (h \in \mathbb{C}^\ell)
$$
with $\|e_k \otimes u_{\alpha_k}\|= \|e_k \otimes u_{\alpha_k}\|_2 =1$.
This also implies that $\varphi$ maps the boundary
$\partial \mathbb{D}^\ell$ of $\mathbb{D}^\ell$ into the boundary
$\partial D =\{T\in L(\mathbb{C}^\ell,K): \|T\|=1\}$ .

Let $D$ be the open unit ball of a spin factor $V$, which is of rank $2$.
Let $e_1$ and $e_2$ be two mutually (triple) orthogonal minimal tripotents
in $V$. Then we have $\|\lambda e_1 + \mu e_2\| = \max\{|\lambda|, |\mu|\}$
for $\lambda, \mu \in \mathbb{C}$ \cite[Corollary 3.1.21]{book}. Hence one can define a continuous map
\begin{equation}
\label{vpl}
\varphi : (z_1,z_2) \in \overline{\mathbb{D}}^2 \mapsto z_1e_1+ z_2e_2 \in \overline D
\end{equation}
 which restricts to an injective holomorphic
map from $\mathbb{D}^2$ to $D$ satisfying $\varphi (0)=0$ and
$\varphi (\partial \mathbb{D}^2) \subset \partial D$.

Given a Hilbert space $H$, a holomorphic map
$f: \mathbb{D}^n \rightarrow H$ admits a power series representation
in terms of homogeneous polynomials from $\mathbb{C}^n$ to $H$
(cf.\,\cite[p.65]{book}). A {\it homogeneous polynomial $p$ of degree $d$}
from $\mathbb{C}^n$ to $H$ is given by
$$
p(z_1, \ldots, z_n) = P(\, (z_1, \ldots, z_n), \ldots,(z_1, \ldots, z_n)\,) \in H,
\qquad (z_1, \ldots, z_n) \in \mathbb{C}^n
$$
where $P :\underbrace{\mathbb{C}^n \times \cdots \times
\mathbb{C}^n}_{d\mbox{-}times} \rightarrow H$ is a $d$-linear map.
Let $\{e_\alpha\}$ be an orthonormal basis in $H$. We can write
$$
p(z_1, \ldots, z_n) = \sum_\alpha \mathfrak{q}_\alpha(z_1, \ldots, z_n)e_\alpha
$$
where $\mathfrak{q}_\alpha(z_1, \ldots, z_n)$ is a homogeneous
polynomial of degree $d$ in $n$ complex variables $z_1, \ldots, z_n$
and has the form
$$
\mathfrak{q}_\alpha(z_1, \ldots, z_n)
= \sum_{j_1+\cdots+j_n=d} c_{\alpha;\, j_1,\ldots,j_n} z_1^{j_1}\cdots z_n^{j_n}
\qquad (c_{\alpha;\, j_1,\ldots,j_n}\in \mathbb{C}).
$$
A holomorphic map $f: \mathbb{D}^n \rightarrow H$ has a power series
representation
$$
f(z_1, \ldots, z_n) = f(0) + \sum_{d=1}^\infty p^d(z_1, \ldots, z_n),
\qquad (z_1, \ldots, z_n)  \in \mathbb{D}^n
$$
where $p^d$ is a homogeneous polynomial of degree $d$ from
$\mathbb{C}^n$ to $H$ and has the from
\begin{equation}
\label{poly}
p^d(z_1, \ldots, z_n) =\sum_\alpha \sum_{j_1+\cdots+j_n=d}
c^d_{\alpha;\, j_1,\ldots,j_n} z_1^{j_1}\cdots z_n^{j_n} e_\alpha
\qquad (c^d_{\alpha;\, j_1,\ldots,j_n}\in \mathbb{C}).
\end{equation}

Let $h: D \rightarrow D'$ be a biholomorphic map between two {\it open
unit balls} $D, D'$ of Banach spaces $V$ and $V'$ respectively. If $h(0)=0$, then it
follows from Cartan's uniqueness theorem that $h$ is the restriction of the derivative
$h'(0): V \rightarrow V'$, which is a linear isometry (cf.\,\cite[Corollary 6]{harris}
and \cite{ku}). In particular, $h$ extends to a continuous map
$\bar h: \bar D \rightarrow \overline D'$ between the closures $\overline D$
and $\overline D'$, where $\bar h = h'(0)|_{\bar D}$. Moreover, $\bar h(\partial D)= \partial D'$.

Let $D$ be a bounded symmetric domain, realised as the open unit ball of
a JB*-triple $V$.  Given a holomorphic embedding $f: D \rightarrow B_H$ of
$D$ into a Hilbert ball $B_H$, the image $f(D)$ is a bounded symmetric domain
and hence there is an equivalent norm $\|\cdot\|_\infty$ on $H$ such that
$(H, \|\cdot\|_\infty)$ is a JB*-triple and $f(D)$ identifies
(via a biholomorphic map) as the open unit ball of
$(H, \|\cdot\|_\infty)$ (cf.\,\cite[Theorem 2.5.26]{book}).
If $f(0)=0$, then the previous remark implies that $f$ extends to a continuous
map $\bar f$, which maps $\partial D$ onto the boundary $\partial f(D)$ of
the domain $f(D)$.

The following lemma is a simple infinite dimensional extension of
Alexander's result in \cite[Proposition 1]{a} (see also \cite[Lemma 1]{k1}).

\begin{lemma}\label{rho}
Let $D$ be a bounded domain with boundary $\partial D$ and $B$ a
Hilbert ball such that the following two continuous maps
$$
\overline{\mathbb{D}}^{\ell}~~ ^{\underrightarrow{\varphi}}
~~ \overline D~~ ^{\underrightarrow{f} }~~ \overline B
$$
on the closures restrict  to holomorphic maps
$$
\mathbb{D}^{\ell}~~ ^{\underrightarrow{\varphi}}
~~ D~~ ^{\underrightarrow{f} }~~ B
$$
with open image $f(D)$, satisfying  $\varphi(\partial \mathbb{D}^\ell) \subset \partial D$
and $f(\partial D) \subset \partial f(D)$.
If $\rho B \subset f(D)$ for some $\rho >0$, then $\ell \rho^2 \leq 1$.
\end{lemma}

\begin{proof}
Let $\{e_\alpha\}$ be an orthonormal basis in the Hilbert space containing
the ball $B$.  By (\ref{poly}), the holomorphic map $f \circ \varphi$ on
$\mathbb{D}^\ell$ has a power series representation
\[
f\circ \varphi (z_1, \ldots, z_{\ell}) =  \sum_{d=1}^\infty p^d(z_1, \ldots, z_{\ell})
\]
where $p^d(z_1, \ldots, z_\ell) $ is a $d$-homogeneous polynomial of the form
$$
p^d(z_1, \ldots, z_\ell) =\sum_\alpha \sum_{j_1+\cdots+j_\ell=d}
c^d_{\alpha;\, j_1,\ldots,j_\ell} z_1^{j_1}\cdots z_\ell^{j_\ell} e_\alpha
\qquad (c^d_{\alpha;\, j_1,\ldots,j_\ell}\in \mathbb{C}).
$$

Since $\rho B \subset f(D)$, we have $\|f(w)\| \geq \rho$ for each
$w\in \partial D$.  Noting that
$f \circ \varphi(\partial \mathbb{D}^\ell) \subset \partial f(D)$, we deduce
\begin{eqnarray*}
\rho^2
&\leq&  \frac{1}{2\pi} \int_0^{2\pi}\|f\circ \varphi(0, \ldots,
e^{i\theta_j}, 0, \ldots, 0)\|^2 d\theta_j\\
&=&  \frac{1}{2\pi}\lim_{r \rightarrow 1}
\int_0^{2\pi}\|f\circ \varphi(0, \ldots,re^{i\theta_j}, 0, \ldots, 0)\|^2 d\theta_j\\
& =& \frac{1}{2\pi}\lim_{r \rightarrow 1} \int_0^{2\pi}\sum_\alpha\left|
\sum_d c^d_{\alpha;\, 0,\ldots, 0,d,0,\ldots,0} r^d e^{id\theta_j}\right|^2d\theta_j\\
&=& \frac{1}{2\pi}\lim_{r \rightarrow 1}\sum_\alpha \int_0^{2\pi}\left|\sum_d
 c^d_{\alpha;\,  0,\ldots, 0,d,0,\ldots,0} r^de^{id\theta_j}\right|^2d\theta_j\\
&=& \lim_{r \rightarrow 1}\sum_\alpha \sum_d \left|c^d_{\alpha;\,  0,\ldots,
0,d,0,\ldots,0}\right|^2 r^{2d}\\
& =& \sum_\alpha \sum_d \left|c^d_{\alpha;\,  0,\ldots, 0,d,0,\ldots,0}\right|^2.
\end{eqnarray*}
It follows that
\begin{eqnarray*}
1 & \geq& \lim_{r\rightarrow 1}\left(\frac{1}{2\pi}\right)^\ell \int_0^{2\pi}
\cdots \int_0^{2\pi} \|f\circ \varphi
(re^{i\theta_1}, \ldots, re^{i\theta_{\ell}}) \|^2 d\theta_1 \cdots d\theta_\ell\\
&=& \lim_{r\rightarrow 1} \sum_\alpha \sum_d\sum_{\nu_1 + \cdots +\nu_\ell=d}
\left|c^d_{\alpha;\,  \nu_1,\ldots, \nu_\ell}\right|^2 r^{2d} \\
&=& \sum_\alpha \sum_d\sum_{\nu_1 + \cdots +\nu_\ell=d} \left|c^d_{\alpha;\,
\nu_1,\ldots, \nu_\ell}\right|^2 \\
&\geq& \sum_\alpha\sum_d \left|c^d_{\alpha;\, d,0,\ldots,0}\right|^2  + \cdots +
\sum_\alpha\sum_d \left|c^d_{\alpha;\,  0,\ldots, 0,d}\right|^2  \geq \ell \rho^2.
\end{eqnarray*}
\end{proof}

In finite dimensions, the squeezing constant of the four series of classical
Cartan domains has been computed by Kubota in \cite{k1}. We will now
compute the squeezing constants of the remaining finite rank bounded
symmetric domains of all dimensions.

We begin with the two exceptional domains which are realised as the open
unit balls of the JB*-triples $M_{1,2}(\mathbb{O})$ and
$H_3(\mathbb{O})$ respectively, where $\dim M_{1,2}(\mathbb{O}) = 16$
and $\dim H_3(\mathbb{O})=27$. Both JB*-triples are equipped with the
spectral norm, as noted previously.
They also carry a Hilbert space structure, with  inner product
\begin{equation}\label{ip}
\langle x, y\rangle = \frac{1}{18}{\rm Trace}\, D(x, y)
\qquad (x,y \in H_3(\mathbb{O})),
\end{equation}
shown in \cite[Corollary 2.14]{roos}, where $D(x,y) = 2 x\bo y$.
Given a minimal tripotent $e\in H_3(\mathbb{O})$, we have
$\langle e, e\rangle =1$ \cite[Proposition 2.8]{roos}.
If $e$ and $u$ are two mutually (triple) orthogonal tripotents in
$H_3(\mathbb{O})$, then $\langle e,u\rangle =0$ \cite[Lemma 2.9]{roos}.

The $27$-dimensional domain $D_{27} \subset H_3(\mathbb{O})$ has
rank $3$ whereas the $16$-dimensional domain
$D_{16} \subset M_{1,2}(\mathbb{O})$ has rank $2$. The following two
propositions, together with Kubota's results in \cite{k1}, give a complete list of
squeezing constants of all finite dimensional irreducible bounded symmetric
domains.

\begin{proposition}\label{27}
The squeezing constant of the exceptional domain $D_{27}$ is given by\\
$\hat\sigma_{D_{27}} = 1/\sqrt 3$.
\end{proposition}

\begin{proof}
We compute $\sigma_{D_{27}} (0) = \hat\sigma_{D_{27}}$.
We have $D_{27} = \{z\in H_3(\mathbb{O}): \|z\|<1\}$, where $\|\cdot\|$
is the spectral norm.
Given $z\in H_3(\mathbb{O})$ with spectral decomposition
\[
z = \alpha_1 e_1 + \alpha_2 e_2 + \alpha_3e_3
\qquad (\alpha_1 \geq \alpha_2 \geq \alpha_3 \geq 0),
\]
the spectral norm $\|z\|$ equals $\alpha_1$, where the minimal tripotents
$e_1, e_2,e_3$ are mutually  orthogonal with respect to the inner product
given in (\ref{ip}). The Hilbert space norm $\|z\|_2$ of $z$ is given by
\[
\|z\|_2^2 = \langle z, z\rangle = \alpha_1^2  +
\alpha_2^2 + \alpha_3^2.
\]
It follows that
\[
\|z\| \leq \|z\|_2 \leq {\sqrt 3} \|z\|
\]
for all $z \in H_3(\mathbb{O})$. This implies
\[
B_{27} \subset D_{27} \subset \sqrt 3 B_{27}
\]
where $B_{27}=\{z\in H_3(\mathbb{O}): \|z\|_2 <1\}$ is the Hilbert ball
in $H_3(\mathbb{O})$. Hence we have
$\hat\sigma_{D_{27}} \geq 1/\sqrt 3$. To show the reverse inequality,
we define a continuous map
$\varphi : \overline{\mathbb{D}}^3 \rightarrow \overline D_{27}$ by
\[
\varphi(z_1, z_2,z_3)
= \left(\begin{matrix} z_1 &0&0\\0&z_2&0\\0&0&z_3
                                 \end{matrix}
                           \right)
= z_1 e_{11} + z_2 e_{22} + z_3 e_{33}
\]
where $e_{jj}$ is the diagonal matrix in $H_3(\mathbb{O})$ with $1$ in the $jj$-entry and $0$
elsewhere. Since $e_{11}, e_{22}, e_{33}$ are mutually (triple) orthogonal
minimal tripotents in $H_3(\mathbb{O})$, we see that $\varphi$
restricts to an injective holomorphic map from $\mathbb{D}^3$ into
$D_{27}$ with $\varphi(0)=0$ and
$\varphi(\partial \mathbb{D}^3 )\subset \partial D_{27}$.
By Lemma \ref{rho} and the remarks before it, for each holomorphic
embedding $f : D_{27} \rightarrow B_{27}$ with $f(0)=0$ and
$\rho B_{27} \subset f(D_{27})$, we must have $3\rho^2 \leq1$.
This proves the reverse inequality.
\end{proof}

\begin{proposition}
\label{16}
The squeezing constant of the exceptional domain $D_{16}$ is given by\\
$\hat\sigma_{D_{16}} = 1/\sqrt 2$.
\end{proposition}

\begin{proof} The arguments are similar to those in the proof of Lemma \ref{27},
we recapitulate for completeness. We consider $M_{1,2}(\mathbb{O})$ as a subtriple of
$H_3(\mathbb{O})$. It suffices to show $\sigma_{D_{16}} (0) =1/\sqrt 2$.
We have $D_{16} = \{z\in M_{1,2}(\mathbb{O}): \|z\|<1\}$, where $\|\cdot\|$
is the spectral norm. Given $z\in M_{1,2}(\mathbb{O})$ with spectral decomposition
\[
z = \alpha_1 e_1 + \alpha_2 e_2  \qquad (\alpha_1 \geq \alpha_2 \geq 0),
\]
the spectral norm $\|z\|$ equals $\alpha_1$, where the minimal tripotents
$e_1, e_2$ are mutually orthogonal with respect to the inner product given in
(\ref{ip}). The Hilbert space norm $\|z\|_2$ of $z$ is given by
\[
\|z\|_2^2 = \langle z, z\rangle = \alpha_1^2  + \alpha_2^2
\]
and
\[
\|z\| \leq \|z\|_2 \leq {\sqrt 2} \|z\|
\]
for all $z \in M_{1,2}(\mathbb{O})$. This implies
\[
B_{16} \subset D_{16} \subset \sqrt 2 B_{16},
\]
where $B_{16}=\{z\in M_{1,2}(\mathbb{O}): \|z\|_2 <1\}$ is the Hilbert ball
in $M_{1,2}(\mathbb{O})$. Hence
$\hat\sigma_{D_{16}} \geq 1/\sqrt 2$. For the reverse inequality, one defines
a continuous map
$\varphi : \overline{\mathbb{D}}^2 \rightarrow \overline D_{16}$ by
\[
\varphi(z_1, z_2) = z_1 e_{11} + z_2 e_{22}
\]
where  $e_{11}=(1,0)$ and $e_{22}=(0,1)$ are mutually (triple) orthogonal
minimal tripotents in $M_{1,2}(\mathbb{O})$,  and $\varphi$
restricts to an injective holomorphic map from $\mathbb{D}^2$ into $D_{16}$
with $\varphi(0)=0$ and $\varphi(\partial \mathbb{D}^2 )\subset \partial D_{16}$.
As before, for each holomorphic embedding $f : D_{16} \rightarrow B_{16}$
satisfying $f(0)=0$ and $\rho B_{16} \subset f(D_{16})$, we must have
$2\rho^2 \leq1$.  This proves the reverse inequality.
\end{proof}

The following result extends Kubota's result \cite{k1} for the classical Cartan
domains to all finite dimensional irreducible bounded symmetric domains.

\begin{corollary}
\label{p}
Let $D$ be a finite dimensional irreducible bounded symmetric domain of rank $p$.
Then its squeezing constant is given by $\hat \sigma_D = 1/\sqrt p$.
\end{corollary}

We are now ready to show that finite-rank bounded symmetric domains, which can be infinite
dimensional, are HHR and compute their squeezing constants.

\begin{theorem} \label{bd}
Let $D$ be a bounded symmetric domain in a complex Banach space.
Then $D$ is HHR if and only if
it is of finite rank. In this case, $D$ is biholomorphic
to a finite product
$$
D_1 \times \cdots \times D_k
$$
of irreducible bounded symmetric domains and we have
\begin{equation}\label{const}
\hat \sigma_D = \left( \frac{1}{\hat \sigma_{D_1}^2} + \cdots
+ \frac{1}{\hat \sigma_{D_k}^2}\right)^{-1/2}.
\end{equation}

If $\dim D_j <\infty$, then  $D_j$ is a classical Cartan domain or an
exceptional domain, and $\hat \sigma_{D_j}= 1/\sqrt{p_j}$
where $p_j$ is the rank of $D_j$.

If $\dim D_j =\infty$, then $D_j$ is either a Lie ball or a Type I domain of
finite rank $p_j$.  For a Lie ball $D_j$, we have $\hat \sigma_{D_j} = 1/\sqrt 2$.
For a rank $p_j$ Type I domain $D_j$, we have $\hat \sigma_{D_j} = 1/\sqrt{p_j}$.

\end{theorem}

\begin{proof} Let $D$ be HHR,
realised as the open unit ball
of a JB*-triple $V$. Then $V$  is linearly homeomorphic to some Hilbert space $H$.  In
particular, $V$ is reflexive and hence $D$
 is of finite rank. Conversely, a finite-rank bounded symmetric domain $D$ decomposes into a finite Cartesian product
$D= D_1 \times \cdots \times D_k$ of irreducible bounded symmetric domains,
where each $D_j$ is of finite rank $p_j$ and realised as the open unit ball of a
Cartan factor $V_j$ for $j=1, \ldots, k$.

To complete the proof, we show that each domain $D_j$ of rank $p_j$ has squeezing constant
$\hat \sigma_{D_j}= 1/\sqrt{p_j}$ and $\hat\sigma_D = (p_1 + \cdots +p_k)^{-1/2}$.

By Corollary \ref{p}, we have $\hat \sigma_{D_j} = 1/\sqrt p_j$ if
$\dim V_j <\infty$. In fact, this is also the case even if
$V_j$ is infinite dimensional, in which case $V_j$ is either a spin factor or the
type I Cartan factor $L(\mathbb{C}^\ell, K)$ with $\dim K = \infty>\ell$.  We now
compute the squeezing constant in these two cases.

First, let $D_j$ be a Lie ball, that is, the open unit ball of a spin factor
$(V, \|\cdot\|)$, which has rank $2$.
In this case, $V$ is a Hilbert space with norm $\|\cdot\|_h$ satisfying
$$
\|\cdot\| \leq \|\cdot\|_h \leq \sqrt 2 \|\cdot\|
$$
(cf.\,\cite[Section 2]{lie}). This gives $\hat\sigma_{D_j} \geq 1/\sqrt 2$.
Making use of the map $\varphi$ in (\ref{vpl}) and analogous arguments in
the proof of Lemma \ref{16}, one concludes that
$\hat\sigma_{D_j} =\hat\sigma_{D_j}(0)= 1/\sqrt 2$.

Next, let $D_j$ be a Type I domain of rank $p_j$, realised as the open unit ball
$$
D_j=\{T\in L(\mathbb{C}^{p_j}, K): \|T\| <1\}
$$
of $L(\mathbb{C}^{p_j}, K)$ with $\dim K = \infty$. Equipped with the
Hilbert-Schmidt norm $\|\cdot\|_2$, the vector space $L(\mathbb{C}^{p_j}, K)$
is a Hilbert space.
Let $B=\{T\in L(\mathbb{C}^{p_j}, K): \|T\|_2 <1\}$ be its open unit ball.
Since $\|\cdot\| \leq \|\cdot\|_2\leq \sqrt{p_j} \|\cdot\|$, we have
$B \subset D_j \subset \sqrt{p_j} B$ and therefore
$\hat\sigma_{D_j}(0) \geq 1/\sqrt {p_j}$.
As before, using the map $\varphi$ in (\ref{vp}) and similar arguments, we
deduce $\hat\sigma_{D_j} = \hat\sigma_{D_j}(0) = 1/\sqrt{p_j}$.

It remains to establish (\ref{const}). The domain
$D =  D_1 \times \cdots\times  D_k$ is the open unit ball of the
$\ell_\infty$-sum
$$
V_1 \oplus \cdots \oplus V_k
$$
of Cartan factors, where  $D_j$ is the open unit ball of $V_j$ of rank $p_j$
for $j=1, \ldots,k$.  We observe from the previous arguments that for each
domain $D_j$,  one can construct a continuous map
$\varphi_j : \overline{\mathbb{D}}^{p_j} \rightarrow \overline D_j$ which restricts
to a holomorphic map from ${\mathbb{D}}^{p_j}$ to $ D_j$ satisfying
$\varphi_j (0) = 0$ and
$\varphi_j(\partial {\mathbb{D}}^{p_j}) \subset \partial D_j$. Hence the product map
$$
\varphi:= \varphi_1 \times \cdots \times \varphi_k :
\overline{\mathbb{D}}^{p_1} \times \cdots \times \overline{\mathbb{D}}^{p_k}
\rightarrow \overline D_1 \times \cdots\times \overline D_k = \overline D
$$
is continuous, which restricts to a holomorphic map from
${\mathbb{D}}^{p_1} \times \cdots \times
{\mathbb{D}}^{p_k}$ to $ D_1 \times \cdots\times  D_k$ satisfying
$\varphi(0, \ldots,0) = (0, \ldots, 0)$ and maps the boundary of
${\mathbb{D}}^{p_1} \times \cdots \times {\mathbb{D}}^{p_k}$ into the boundary of
$ D_1 \times \cdots\times  D_k =D$. Applying Lemma \ref{vp} again, we deduce
that
$$
\hat\sigma_D \leq \frac{1}{\sqrt{p_1 + \cdots+p_k}}
= \frac{1}{\sqrt{\hat{\sigma}_{D_1}^{-2} + \cdots + \hat{\sigma}_{D_k}^{-2}}}.
$$

For each $j=1, \ldots,k$, the previous arguments reveal that
there is a Hilbert space $H_j$ with open unit ball $B_j$ such that
$$
B_j \subset D_j \subset \sqrt{p_j} B_j.
$$
Let $B$ be the open unit ball of the Hilbert space direct sum
$H_1 \oplus_2 \cdots \oplus_2 H_k$.
Then we have
$$
B \subset D_1 \times \cdots \times D_k \subset \sqrt{ p_1 + \cdots + p_k}\,B.
$$
This implies that
$$
\hat\sigma_D \geq \frac{1}{\sqrt{p_1 + \cdots+p_k}}
$$
which completes the proof.
\end{proof}

\section{Uniformly elliptic domains}\label{u}

\newcommand{\CC}{\mathbb{C}}
\newcommand{\cH}{H}

Bounded symmetric domains can be realised as convex domains in Banach spaces and those which are
HHR have been completely determined previously. We conclude the paper in this section by introducing
a large class of bounded convex domains, which include the strongly convex domains, and show that these
domains are HHR in Hilbert spaces.
The domains to be introduced are called {\it uniformly elliptic domains}.

We begin with a preamble. Recall that
a finite dimensional bounded domain $D\subset \mathbb{C}^n$ with a $C^2$ boundary $\partial D$
is called {\it strongly convex} if all normal curvatures of $\partial D$ are positive (cf.\,\cite[p.108]{ab}).
Such a domain is a manifold with curvature pinched which entails
the existence of two positive constants $R>r>0$ such that for each $q\in \partial D$,
there are two points $q',q''$ in $\mathbb{C}^n$ with the property that $q$ is a common boundary point of the Euclidean balls
$B_{\mathbb{C}^n} (q', r)$ and $B_{\mathbb{C}^n} (q'', R)$ satisfying $B_{\mathbb{C}^n} (q', r)
\subset D \subset B_{\mathbb{C}^n} (q'', R)$. For fixed $r$ and $R$, it can be seen that $q'$ and $q''$
are unique  and colinear with $q$. For instance, an ellipsoid is strongly convex and has this property.

In view of the fact that Hilbert balls are the only bounded symmetric domains with a $C^2$ boundary,
we generalise the concept of strong convexity to infinite dimension without the assumption of a smooth
boundary, to cover a wider class of domains, as follows.

\begin{definition}\label{sc} \rm
A bounded  convex domain $\Omega$ in a complex Banach space
$V$   is called
{\it uniformly elliptic} if there exist universal constants $r, R$ with $0<r< R$
such that to each $q \in \partial\Omega$, there correspond two unique points $q', q'' \in V$,
colinear to $q$, satisfying
\begin{itemize}
\item[(\ref{sc}.1)] $B_V (q', r) \subset \Omega \subset B_V (q'', R)$;
\item[(\ref{sc}.2)] $q\in \partial B_V (q', r) \cap \partial B_V (q'', r)$, that is,
$q$ is a common boundary point of $ B_V (q', r)$,
$B_V(q'', R)$ and $\Omega$.
\end{itemize}
\end{definition}

Evidently,
the definition of uniform ellipticity depends on the norm of the ambient Banach space.
By the previous remarks, strongly convex domains are uniformly elliptic, but the converse is
false. In fact,  all open
balls in Banach spaces are uniformly elliptic. Indeed, if say, $\Omega = B_V$ is the open unit ball of
a Banach space $V$, then for each boundary point $q \in \partial \Omega$,
we have $\|q\|=1$ and
$$B_V(q/2, {1/2}) \subset \Omega = B_V(0,1)$$
and $q \in \partial B_V(\frac{q}{2},\frac{1}{2})\, \cap\, \partial \Omega \, \cap\,\partial  B_V(0,1)$.
For $R=1$ and $r=1/2$, the points $q'=q/2$ and $q'' = 0$ are unique and colinear to $q$.

By definition, each point $p$ in a uniformly elliptic domain $\Omega$ in a Banach space $V$
lies in the ball $B_V(q'',R)$ for all $q\in
\partial \Omega$, as in (\ref{sc}.1) above, although $p$ need not be colinear with $q$ and $q''$.
We consider the question of colinearity below.

\begin{lemma}\label{lin} Let $\Omega$ be a uniformly elliptic domain in a Banach space $V$ and for each
$q\in \partial \Omega$, let
$$B_V (q', r) \subset \Omega \subset B_V (q'', R), \quad q\in \partial B_V (q', r) \cap \partial B_V (q'', r)$$
be as in the definition of uniform ellipticity. Then for each $p\in \Omega$ and $q\in \partial \Omega$
with $\|p-q\| = d(p, \partial \Omega)$, either
$p$ is colinear with $q$ and $q''$ or, there exists $q_1\in \partial \Omega$ such that $p$ is
colinear with $q_1$ and $q_1''=q''$ satisfying $\|p-q_1\|=\|p-q\|$.
\end{lemma}

\begin{proof} Let $q\in \partial \Omega \cap \partial B_V(q'',R)$ satisfy $\|p-q\|= d(p, \partial \Omega)$.
Suppose $p$ is not colinear with $q$ and $q''$. We show the existence of $q_1$ in the lemma.

Consider $p\in \Omega \subset B_V(q'',R)$. Extend the (real) line through $q''$ and $p$ to a point
$q_1 \in \partial B_V(q'',R)$. Then we have $\|p-q_1\| = d(p, \partial B_V(q'',R))\leq \|p-q\|$. We show that
$q_1 \in \partial \Omega$,  which would imply $\|p-q_1\| \geq \|p-q\|$
and complete the proof by uniqueness of
$q_1'$ and $q_1''$.

If $q_1 \notin \partial \Omega$, we deduce a contradiction. Since $q_1 \notin \overline\Omega$ and $p\in \Omega$,
the line joining $p$ and $q_1$ must intersect $\partial \Omega$ at some point $\omega$, say. Now
we have the contradiction
$$\|p-q\| \geq  \|p-q_1\| > \|p-\omega\| \geq d(p, \partial \Omega)=\|p-q\|.$$
\end{proof}

We will discuss uniformly elliptic domains in greater detail in another work, but complete this section
presently by showing that these domains are HHR in Hilbert spaces, which
generalises the finite dimensional result for strongly convex domains in \cite[Proposition 1]{yeung}.

\begin{theorem}\label{strongconv}
Let $\Omega$ be a uniformly elliptic domain in a
Hilbert space $H$.  Then $\Omega$ is HHR.
\end{theorem}

\begin{proof} \rm
We need to show that the squeezing function $\sigma_\Omega$ of $\Omega$ has a strictly positive lower
bound. Suppose, to the contrary, that there is a sequence $(p_\nu)$ in $\Omega$ such that
\begin{equation}\label{zero}
\lim_{\nu\rightarrow \infty} \sigma_\Omega(p_\nu) =0.
\end{equation}
We deduce a contradiction.
By Lemma \ref{boundary}, we may assume, by choosing another sequence if necessary,
that $ d(p_\nu, \partial \Omega)$ converges to $0$ as $\nu \rightarrow \infty$
and one can find a boundary point  $q_\nu \in \partial \Omega$ such that
 $$\| q_\nu -  p_\nu\|=d( p_\nu, \partial \Omega)> 0.$$

Write $\lambda_\nu = d( p_\nu, \partial \Omega)$ and let
$$B_V (q_\nu', r) \subset \Omega \subset B_V (q_\nu'', R), \quad
q_\nu\in \partial B_V (q_\nu', r) \cap \partial B_V (q_\nu'', R)$$
be as in the definition of uniformly ellipticity of $\Omega$ where, by Lemma \ref{lin},
$q_\nu$ can be chosen so that $p_\nu$ lies on
the line through $q_\nu$ and $q_\nu''$.

We complete the proof by a contradiction that there is
a subsequence $(p_{\nu'})$ of $(p_\nu)$ and a constant
$\delta > 0$ satisfying
\[
\sigma_\Omega (p_{\nu'}) > \delta \quad \mbox{for all } \nu'.
\]
In fact, $\delta$ depends only on $r$ and $R$.

For each  $\nu$, we define a holomorphic embedding
$$\Phi\circ L_{\nu} : \Omega \rightarrow H$$
as follows.
Let $\mathbf{e^1}$ be the unit vector
\[
\mathbf{e^1} :=  \frac{q_\nu'' -q_\nu}{\|q_\nu''-q_\nu\|}.
\]
We have
\begin{itemize}
\item[(S1)] $q_\nu'' = R\mathbf{e}^1+q_\nu, ~q_\nu'=r\mathbf{e}^1 + q_\nu$,
\item[(S2)] $p_\nu = \lambda_\nu \mathbf{e}^1+q_\nu~$
($\lambda_\nu \rightarrow 0$ as $\nu \rightarrow \infty$).
\end{itemize}
Since $\sigma_{\Omega}(p_\nu) = \sigma_{\Omega-q_\nu}(p_\nu-q_\nu)$,
taking a translation, we may assume $q_\nu =0$. Then we have
\begin{itemize}
\item[(S1$^0$)] $ q_\nu'' = R \mathbf{e}^1, ~\vp q_\nu'=r \mathbf{e}^1$,
\item[(S2$^0$)] $ p_\nu = \lambda_\nu \mathbf{e^1} $.
\end{itemize}
We now have
\begin{equation}\label{R}
p_\nu = \lambda_\nu \mathbf{e}^1\in   B_H(r\mathbf{e}^1,r) \subset
\Omega \subset
  B_H(R\mathbf{e}^1, R).
\end{equation}
where
$$q_\nu' =  r \mathbf{e}^1, \quad
 q_\nu'' =  R \mathbf{e}^1.$$

Extend  $\{\mathbf{e}^1\}$ to an orthonormal
basis $\{\mathbf{e}^\gamma\}_{\gamma \in \Gamma}$ in $H$.
For each $z\in H$, we will write
$$z = \sum_{\gamma \in \Gamma} z_\gamma \mathbf{e}^\gamma = z_1 \mathbf{e}^1
+\sum_{\gamma \neq 1} z_\gamma \mathbf{e}^\gamma$$
with $z_\gamma \in \CC$.
We have $$z \in  B_H(r\mathbf{e}^1,r) \Leftrightarrow \| z -r\mathbf{e}^1\|<r$$
where
\begin{equation}\label{V1}
\| z -r\mathbf{e}^1\|^2 =|z_1 -r|^2 + \sum_{\gamma \neq 1} |z_\gamma|^2
= |z_1|^2 - 2r{\rm Re}\,z_1 + r^2
+ \sum_{\gamma \neq 1} |z_\gamma|^2.
\end{equation}

We definite a dilation $L_\nu : H \rightarrow H$  by
\[
L_\nu (z) = \frac{z_1}{ \lambda_\nu} \mathbf{e}^1
+ \frac{1}{\sqrt{\lambda_\nu}}\sum_{\gamma\neq 1}z_\gamma\mathbf{e}^\gamma,
\quad z= \sum_{\gamma \in \Gamma} z_\gamma \mathbf{e}^\gamma
\]
which satisfies $L_{\nu}(p_\nu) = \mathbf{e}^1$.  The map $L_\nu$ is a linear
homeomorphism of $\cH$, with inverse
$$L_\nu^{-1}(z) =  \lambda_\nu z_1 \mathbf{e}^1 + \sqrt{\lambda_\nu}\sum_{\gamma\neq1}z_\gamma\mathbf{e}^\gamma.$$

Define a Cayley transform $\Phi : \{z\in H : {\rm Re}\, z_1 > 0\} \rightarrow H$ by
\[
 \Phi (z) := \frac{z_1-1}{z_1+1} \mathbf{e}^1
+ \sum_{\gamma\neq 1}\frac{\sqrt{2}z_\gamma}{z_1+1} \mathbf{e}^\gamma, \quad z= \sum_\gamma z_\gamma \mathbf{e}^\gamma
\]
and the holomorphic embedding $$\Phi\circ L_{\nu} : \Omega \rightarrow H$$ where
$\Phi(L_{\nu}( p_\nu))=0$. Although $\Phi$ depends on $\nu$, we omit the subscript $\nu$ indicating
this, to simplify notation, since confusion is unlikely.

We will show that
$$B_H\left(0, \sqrt {\frac{r}{2+2r}}\,\right) \subset \Phi(L_{{\nu}}( \Omega ))\subset B_H(0,\sqrt{1+R})$$
for sufficiently large $\nu$.

Substituting $R$ for $r$ in (\ref{V1}), we see that
$$B_H( R \mathbf{e}^1,  R)= \{z\in H: \|z - R\mathbf{e}^1\|^2  <  R^2 \}
=\{z\in H: \sum_{\gamma\in \Gamma}|z_\gamma|^2  < 2 R \text{ Re } z_1\}.$$
Given $\zeta = \sum_\gamma \zeta_\gamma \mathbf{e}^\gamma \in B_H$, we have
$$\Phi^{-1}(\zeta) = \frac{1+\zeta_1}{1-\zeta_1}\mathbf{e}^1
+ \sum_{\gamma\neq 1}\frac{\sqrt 2 \zeta_\gamma}{1-\zeta_1}\mathbf{e}^\gamma.
$$
Hence
\begin{eqnarray*}
&&\zeta \in \Phi L_{\nu}(B_H( R \mathbf{e}^1,R)) \Leftrightarrow  L_{\nu}^{-1}\Phi^{-1}\zeta \in
B_H( R \mathbf{e}^1, R)\\
&\Leftrightarrow& \lambda_{\nu} \left(\frac{1+ \zeta_1}{1-\zeta_1}\right)\mathbf{e}^1
+\sum_{\gamma\neq 1} \frac{\sqrt {2 \lambda_\nu}}{1-\zeta_1}\,\mathbf{e}^\gamma \in B_H( R \mathbf{e}^1, R)\\
&\Leftrightarrow& \alpha\lambda_\nu^2 |1+\zeta_1|^2 + 2\lambda_\nu \sum_{\gamma\neq1}|\zeta_\gamma|^2
< 2 R\lambda_\nu (1-|\zeta_1|^2)\\
&\Rightarrow &  \sum_{\gamma\neq1}|\zeta_\gamma|^2 <  R (1-|\zeta_1|^2)\\
&\Rightarrow & \|\zeta\|^2=|\zeta_1|^2 +  \sum_{\gamma\neq1}|\zeta_\gamma|^2 < 1+ R
\end{eqnarray*}
and therefore we have, by (\ref{R}),
\begin{equation}\label{large}
\Phi L_{{\nu}}(\Omega) \subset
\Phi L_{{\nu}}(B_H( R \mathbf{e}^1, R))\subset B_H(0, \sqrt{1+R}).
\end{equation}

We now show that $B_H(0,\sqrt {\frac{r}{2+2r}}\,) \subset \Phi L_{{\nu}} (\Omega)$ for sufficiently large $\nu$.
For this, we will make use of the inclusion $B_H(r\mathbf{u},r)\subset \Omega$.

We have
$L_{{\nu}}^{-1}\Phi^{-1}(\zeta) \in B_H(r\mathbf{u},r)$ if and only if
$\|L_{{\nu}}^{-1}\Phi^{-1}(\zeta)  -r\mathbf{u}\|<r$, where
\begin{eqnarray}\label{fin}
&&\|L_{\nu}^{-1} \Phi^{-1}(\zeta) -r\mathbf{u}\|^2 < r^2
\Leftrightarrow \left|\lambda_{\nu}\left(\frac{1+\zeta_1}{1-\zeta_1}\right) -r\right|^2
+\frac{2\lambda_{\nu}}{|1-\zeta_1|^2} \sum_{\gamma \neq 1} |\zeta_\gamma |^2
<r^2\\\nonumber
&\Leftrightarrow& \lambda_\nu(\lambda_\nu|1+\zeta_1|^2- 2r(1-|\zeta_1|^2)
 + 2\sum_{\gamma \neq 1} |\zeta_\gamma|^2) <0.
\end{eqnarray}
For $\zeta \in B_H(0, \sqrt {\frac{r}{2+2r}}\,)$, we have $2r -(2r|\zeta_1|^2 +2\|\zeta\|^2) > r$ and
$|1+\zeta_1|^2 \leq \displaystyle \left(1+ \sqrt {\frac{r}{2+2r}}\,\right)^2$. Since $\lambda_\nu \rightarrow 0$
as $\nu \rightarrow \infty$, there exists ${\nu_0}$ such that $\nu \geq \nu_0$ implies
$$\lambda_\nu < \frac{r}{2\left(1+ \sqrt{\frac{r}{2+2r}}\right)^2}$$
and hence
\begin{eqnarray*}
&&\lambda_\nu|1+\zeta_1|^2- 2r(1-|\zeta_1|^2)
 + 2\sum_{\gamma \neq 1} |\zeta_\gamma|^2
  \leq \lambda_\nu|1+\zeta_1|^2- 2r + 2r|\zeta_1|^2 + 2\|\zeta\|^2\\
&& < \frac{r}{2\left(1+ \sqrt{\frac{r}{2+2r}}\right)^2}|1+\zeta_1|^2 - r < -r/2
\end{eqnarray*}
which gives $\lambda_\nu(\lambda_\nu|1+\zeta_1|^2- 2r(1-|\zeta_1|^2)
 + 2\sum_{\gamma \neq 1} |\zeta_\gamma|^2) <- r\lambda_\nu/2 <0$ and  by (\ref{fin}),
$$ \|L_{\nu}^{-1} \Phi^{-1}(\zeta) -r\mathbf{u}\|^2 < r^2.$$
We have therefore shown that, for  $\nu \geq \nu_0$,
the inclusions
$$B_H(0, \sqrt {r/(2+2r}\,)\subset \Phi L_{\nu}(B_H(r\mathbf{e}^1,r))
\subset \Phi L_{\nu}(\Omega)$$ are satisfied.

Now it follows from this and (\ref{large}) that
\[
\sigma_\Omega (p_{{\nu}})  \ge\sqrt{\frac{ r}{2(1+r)(1+R)}} \,>0
\]
for all $\nu \geq {\nu_0}$, which contradicts
$\lim_{\nu\rightarrow \infty} \sigma_\Omega(p_{{\nu}}) =0$
and completes the proof.

\end{proof}


\begin{thebibliography}{99}

\bibitem{ab} M. Abate, Interation of holomorphic maps on taut manifolds, Mediterranean Press,
Consenza (1989)

\bibitem{a} H. Alexander, {\it Extremal holomorphic embeddings between the
ball and polydisc}, Proc. Amer. Math. Soc. \textbf{68}, 200--202 (1978)

\bibitem{b} J.M. Borwein and S. Fitzpatrick, {\it Existence of nearest points in
Banach spaces}, Can. J. Math. \textbf{41}, 702-720 (1989)

\bibitem{book} C-H. Chu, Jordan structures in geometry and analysis,
Cambridge Univ. Press, Cambridge (2012)

\bibitem{lie} C-H. Chu, \textit{Iteration of holomorphic maps on Lie balls},
Adv. Math. \textbf{264}, 114--154 (2014)

\bibitem{cb} C-H. Chu and L. J. Bunce, {\it Compact operations, multipliers
and Radon-Nikodym property in JB*-triples}
Pacific  J. Math. \textbf{153}, 249--265 (1992)

\bibitem{deng} F. Deng, Q. Guan and L. Zhang, {\it Some properties of
squeezing functions on bounded domains}, Pacific J. Math. \textbf{257},
319--341 (2012)

\bibitem{deng1} F. Deng, Q. Guan, and L. Zhang, {\it Properties of squeezing
functions and global transformations of bounded domains}, Trans. Amer. Math.
Soc. \textbf{368}, 2679--2696 (2016)

\bibitem{fo}  J. E. Fornaess and F. Rong, {\it Estimate of the squeezing function for a class of bounded
	domains}, arXiv:1606.01335 (2016)

\bibitem{fv} T. Franzoni and E. Vessentini, {\it Holomorphic maps and invariant
distances}, Math. Studies 40,  North-Holland, Amsterdam (1980)

\bibitem{harris} L.A. Harris, {\it Schwarz's lemma in normed linear spaces},
Proc. Nat. Acad. Sci. USA. \textbf{62}, 1014--1017 (1969)

\bibitem{kaup2} W. Kaup,  {\it A Riemann mapping theorem for bounded
symmetric domains in complex Banach spaces}, Math. Z. {\bf 183}, 503--529
 (1983)

\bibitem{kaup1} W. Kaup, {\it On a Schwarz lemma for bounded symmetric
domains}, Math. Nachr. \textbf{197}, 51--60 (1999)

\bibitem{ku} W. Kaup and H. Upmeir, {\it Banach spaces with
biholomorphically equivalent unit balls are isomorphic}, Proc. Amer.
Math. Soc. \textbf{58}, 129--133 (1976)

\bibitem{kim} K-T. Kim and L. Zhang, {\it On the uniform sequeezing
property of bounded convex domains in $\mathbb{C}^n$},
Pacific J. Math. \textbf{282}, 341--358 (2016)

\bibitem{ko1}  S. Kobayashi, {\it Intrinsic distances, measure and geometric function
theory}, Bull. Amer. Math. Soc. \textbf{82}, 357--416 (1976)

\bibitem{ko} S. Kobayashi, {\it Hyperbolic manifolds and holomorphic
mappings}, World Scientific, Singapore (2005)

\bibitem{k1} Y. Kubota, {\it A note on holomorphic imbeddings of the
classical Cartan domains into the unit ball}, Proc. Amer. Math. Soc.
\textbf{85}, 65--68 (1982)

\bibitem{kw} S. Kwapien, {\it Isomorphic characterizations of Hilbert spaces
by orthogonal series with vector valued coefficients},
S\'eminaire d'analyse fonctionnelle (Polytechnique)
exp. no.8, 1--7 (1972--1973)

\bibitem{lsy} K. Liu, X. Sun and S.-T. Yau, {\it Canonical metrics on the
moduli space of Riemann surfaces I}, J. Diff. Geom. \textbf{68}, 571--637 (2004)

\bibitem{lsy1}  K. Liu, X. Sun and S.-T. Yau, {\it Canonical metrics on the
moduli space of Riemann surfaces II}, J. Diff. Geom. \textbf{69}, 163--216 (2005)

\bibitem{muji} J. Mujica, Complex analysis in Banach spaces, Math. Studies 120, North Holland,
Amsterdam (1986)

\bibitem{pel} A. Pelczynski and C. Bessaga, {\it Some aspects of the present
theory of Banach spaces}, In  S. Banach: Travaux sur L'Analyse Fonctionnelle,
Warszaw (1979)

\bibitem{roos} G.J. Roos, {\it Exceptional symmetric domains},
Symmetries in complex analysis, Contemp. Math. \textbf{468}, 157-189 (2008)


\bibitem{yeung} S-K. Yeung, {\it Geometry of domains with the uniform
squeezing property}, Adv. Math. \textbf{221}, 547--569 (2009)

\end{thebibliography}
\end{document}